\documentclass[a4paper,12pt]{amsart}
\usepackage{amsmath}
\usepackage{amssymb}
\usepackage{mathrsfs}
\usepackage{enumerate}
\usepackage{ifthen}
\usepackage{graphicx}
\usepackage[T1]{fontenc} %skandit
\usepackage{tabularx}
\usepackage{multirow}
\usepackage{color,soul}
\nonstopmode \numberwithin{equation}{section}
\newtheorem{thm}{Theorem}[section]

\newtheorem{lem}{Lemma}[section]

\theoremstyle{definition}

\newcounter{minutes}
\divide\time by 60
\newcounter{hours}
\multiply\time by 60
\addtocounter{minutes}{-\time}
\newcounter {own}
\def\theown {\thesection       .\arabic{own}}

{\qed\bigskip}

\newcounter{alphabet}

\begin{document}

\title{On a certain class of Starlike Functions}
\author{Md Firoz Ali}
\address{Md Firoz Ali,
National Institute of Technology Durgapur,
West Bengal-713209, India}
\email{ali.firoz89@gmail.com, fali.maths@nitdgp.ac.in}

\author{Md Nurezzaman}
\address{Md Nurezzaman,
National Institute of Technology Durgapur,
West Bengal-713209, India}
\email{nurezzaman94@gmail.com}

\subjclass[2020]{Primary 30C45, 30C55}
\keywords{analytic functions; univalent functions; starlike functions; convex functions; successive coefficients; Hankel determinant}

\def\thefootnote{}
\footnotetext{ {\tiny File:~\jobname.tex,
printed: \number\year-\number\month-\number\day,
          \thehours.\ifnum\theminutes<10{0}\fi\theminutes }
} \makeatletter\def\thefootnote{\@arabic\c@footnote}\makeatother

\begin{abstract}
Let $\mathcal{S}_u^*$ denote the class of all analytic functions $f$ in the unit disk $\mathbb{D}:=\{z\in\mathbb{C}:|z|<1\}$, normalized by $f(0)=f'(0)-1=0$ that satisfies the inequality $\left|zf'(z)/f(z)-1\right|<1$ in $\mathbb{D}$. In the present article, we obtain the sharp estimate of Hankel determinants whose entries are coefficients of $f\in\mathcal{S}_u^*$, logarithmic coefficients of $f\in\mathcal{S}_u^*$ and  coefficients  of inverse of $f\in\mathcal{S}_u^*$, respectively. We also obtain, the sharp estimate of the successive coefficients for functions in  the class $\mathcal{S}_u^*$.
\end{abstract}

\thanks{}

\maketitle
\pagestyle{myheadings}
\markboth{Md Firoz Ali and Md Nurezzaman }{On a certain class of starlike functions}

\section{Introduction}
Let $\mathcal{H}$ be the class of all analytic functions in the open unit disk $\mathbb{D}:=\{z\in\mathbb{C}:z|<1\}$ and $\mathcal{A}$ be the class of functions $f\in\mathcal{H}$ with the normalization $f(0)=f'(0)-1=0$, i.e., having the Taylor series expansion
\begin{equation}\label{Q-01}
f(z)= z+\sum_{n=z}^{\infty}a_n z^n.
\end{equation}
The subclass $\mathcal{S}$ of $\mathcal{A}$ consisting of univalent functions (that is, one-to-one) in $\mathbb{D}$ has attracted much more interest for over a century, and is a central area of research in complex analysis.
A function $f\in\mathcal{A}$ is called starlike (respectively, convex) if $f(\mathbb{D})$ is a starlike domain with respect to the origin (respectively, convex domain). The set of all starlike functions and convex functions in $\mathcal{S}$ are denoted by $\mathcal{S}^*$ and $\mathcal{C}$, respectively. It is well-known that a function $f\in\mathcal{A}$ is starlike (respectively, convex) if and only if ${\rm Re\,} \left(zf'(z)/f(z)\right)>0$ (respectively, ${\rm Re\,} \left(1+zf''(z)/f'(z)\right)>0$) for $z\in\mathbb{D}$. For further information about these classes we refer to \cite{1983-Duren, 1983-Goodman}.\\
%A function $f\in \mathcal{A}$ is close-to-convex if and only if there exists a starlike function $g\in \mathcal{S}^*$ and a real number $\alpha\in(-\pi/2,\pi/2)$ such that (see \cite{1983-Duren, 1952-Kaplan})
%\begin{align*}
%{\rm Re\,} \left(e^{i\alpha}\frac{zf'(z)}{g(z)}\right)>0,\quad z\in\mathbb{D}.
%\end{align*}
%The class of all close-to-convex function is denoted by $\mathcal{K}$. For $g(z)=z$ and $\alpha=0$ in the class $\mathcal{K}$, we have the class $\mathcal{R}$ of bounded turning functions defined as
%$$
%\mathcal{R}=\left\{f\in\mathcal{A}:{\rm Re\,} f'(z)>0,\quad z\in\mathbb{D}\right\}.
%$$

For two functions $f$ and $g$ in $\mathcal{H}$, we say that $f$ is subordinate to $g$, written as $f\prec g$ if there exists a function $\omega\in\mathcal{H}$ with $\omega(0)=0$ and $|\omega(z)|<1$ for $z\in\mathbb{D}$ such that $f(z)=g(\omega(z))$ in $\mathbb{D}$. Let $\varphi$ be an analytic univalent function with positive real part such that $\varphi(\mathbb{D})$ is symmetric with respect to the real axis and starlike with respect to $\varphi(0) = 1$ and $\varphi'(0)>0$.
%Let $\mathcal{P}$ be the subclass of $\mathcal{H}$ that consist of function $p$ in $\mathbb{D}$ such that $p(0)=1$ and $\mathrm{Re}~ p(z)>0$. A function $p\in\mathcal{P}$ has the Taylor series expansion
%\begin{align}\label{Q-05}
%p(z)=1+\sum\limits_{n=1}^{\infty}p_nz^n,\quad z\in\mathbb{D}.
%\end{align}
%Let $\mathcal{Q}$ be the subclass of $\mathcal{P}$ containing function $\varphi$ such that $\varphi(\mathbb{D})$ is symmetric with respect to the real axis and starlike with respect to $1$.
For such $\varphi$, Ma and Minda \cite{1992-Ma-Minda} introduced the classes $\mathcal{S}^*(\varphi)$ and $\mathcal{C}(\varphi)$ as
\begin{align*}
\mathcal{S}^*(\varphi)=\left\{f\in\mathcal{A}:\frac{zf'(z)}{f(z)}\prec\varphi(z)\right\}
\end{align*}
and
\begin{align*}
\mathcal{C}(\varphi)=\left\{f\in\mathcal{A}:1+\frac{zf''(z)}{f'(z)}\prec\varphi(z)\right\}.
\end{align*}
One can easily find that $f\in\mathcal{C}(\varphi)$ if and only if $zf'\in\mathcal{S}^*(\varphi)$. For different choices of $\varphi$, several well-known classes can be easily obtained from these classes which were earlier studied due to their geometric and analytic properties. For instance, $\mathcal{S}^*((1+(1-2\alpha)z)/(1-z)):=\mathcal{S}^*(\alpha)$ and $\mathcal{C}((1+(1-2\alpha)z)/(1-z)):=\mathcal{C}(\alpha)$ are the classes of starlike and convex functions of order $\alpha$, $0\le\alpha<1$, respectively. In particular, $\mathcal{S}^*(0)=\mathcal{S}^*$ and $\mathcal{C}(0)=\mathcal{C}$.\\
%Now for a different choice of the function $\varphi$ in the above class, we get many subclasses of $\mathcal{A}$. Particularly, if we choose $(1+z)/(1-z)~\text{and}~(1+(1-2\alpha)z)/(1-z)$ in place of $\varphi$ then the class $\mathcal{S}^*(\varphi)$ reduces to the class $\mathcal{S}^*$ of starlike function, the class $\mathcal{S}^*(\alpha)$ of starlike function of order $\alpha$ with $0\le\alpha<1$ due to Robertson \cite{1936-Robertson}. Also note that $\mathcal{S}^*(0)=\mathcal{S}^*$. On the other hand, $\mathcal{C}\left((1+z)/(1-z)\right)$ and $\mathcal{C}\left((1+(1-2\alpha)z)/(1-z)\right)$ is the class $\mathcal{C}$ of convex function and the class $\mathcal{C}(\alpha)$ convex function of order $\alpha$ for $0\le\alpha<1$.\\
% A function $f\in\mathcal{A}$ is said to be close-to-convex if the complement of the image-domain $f(\mathbb{D})$ in $\mathbb{C}$ is the union of rays that are disjoint (except that the origin of one ray may lie on another one of the rays) and the class of all close-to-convex functions is denoted by $\mathcal{K}$, which was first introduced by Kaplan \cite{1952-Kaplan}. A function $f\in \mathcal{K}$ if there exists a starlike function $g$ such that (see \cite{1983-Duren, 1952-Kaplan})
%\begin{align*}
%{\rm Re\,} \left(e^{i\alpha}\frac{zf'(z)}{g(z)}\right)>0,\quad -\frac{\pi}{2}<\alpha<\frac{\pi}{2},~z\in\mathbb{D}.
%\end{align*}\\

In the present article, we are interested in the class $\mathcal{S}_u^*$, consisting of functions $f\in\mathcal{A}$ such that
\begin{align}\label{Q-10}
\left|\frac{zf'(z)}{f(z)}-1\right|<1\quad \text{for}~z\in\mathbb{D}.
\end{align}
The class $\mathcal{S}_u^*$ was introduced and studied by Singh \cite{1968-Singh}. It is easy to see that every function in $\mathcal{S}_u^*$ is also belongs to $\mathcal{S}^*$ and $\mathcal{S}_u^*=\mathcal{S}^*(\varphi)$ for $\varphi(z)=1+z$. Singh \cite{1968-Singh,1974-Singh-Singh} obtained the sharp distortion theorem, estimate of the Fekete-Szeg\"{o} functional and radius of convexity for functions in $\mathcal{S}_u^*$. Further, the author in \cite{1968-Singh} proved that if $f\in \mathcal{S}_u^*$, then
\begin{align}\label{Q-15}
|a_n|\le \frac{1}{n-1} ~~\text{for all}~~n\ge 2
\end{align}
and the equality holds for the functions $f_n(z)=ze^{\frac{z^{n-1}}{n-1}}$, $n=2,3,\cdots$.\\

For $q,n\in\mathbb{N}$, Pommerenke \cite{1966-Pommerenke, 1967-Pommerenke} defined the $q$-th order Hankel determinant for functions $f$ of the form $f(z)= \sum_{n=1}^{\infty}a_n z^n$ as
$$
H_{q,n}(f)=
\begin{vmatrix}
a_n&a_{n+1}& \cdots &a_{n+q-1}\\
a_{n+1}&a_{n+2}& \cdots &a_{n+q}\nonumber\\
\vdots&\vdots& \cdots & \vdots\nonumber\\
a_{n+q-1}& a_{n+q}& \cdots & a_{n+2(q-1)}
\end{vmatrix}.
$$
In particular, for $f\in \mathcal{A}$, one can easily see that $H_{2,1}(f)=a_3-a_2^2$,
\begin{align}\label{Q-20}
H_{2,2}(f)=\begin{vmatrix}
a_2 & a_3\\
a_3 &a_4
\end{vmatrix}=a_2a_4-a_3^2
\end{align}
and
\begin{align}\label{Q-25}
H_{3,1}(f)=
\begin{vmatrix}
a_1 & a_2 & a_3\\
a_2 & a_3 & a_4\\
a_3 & a_4 & a_5
\end{vmatrix}
=(a_3a_5-a_4^2)-a_2(a_2a_5-a_3a_4)+a_3(a_2a_4-a_3^2).
\end{align}
The logarithmic coefficients $\gamma_n$ of a function $f\in \mathcal{A}$ of the form \eqref{Q-01} are defined by
\begin{align}\label{Q-30}
\mathrm{F}_f(z)=\log\frac{f(z)}{z}=2\sum\limits_{n=1}^{\infty}\gamma_nz^n.
\end{align}
%As we have,
%\begin{align}\label{Q-150}
%\log\frac{f(z)}{z}=&\log{\left(1+\sum\limits_{n=2}^{\infty}a_nz^n\right)}\nonumber\\
%=&\sum\limits_{n=2}^{\infty}a_nz^{(n-1)}-\frac{\left(\sum\limits_{n=2}^{\infty}a_nz^{(n-1)}\right)^2}{2}+\frac{\left(\sum\limits_{n=2}^{\infty}a_nz^{(n-1)}\right)^3}{3}+\cdots
%\end{align}
Equating the coefficients of $z$, $z^2$ and $z^3$ in both sides, we obtain
\begin{align}\label{Q-35}
\gamma_1=\frac{a_2}{2},\quad \gamma_2=\frac{1}{2}\left(a_3-\frac{a_2^2}{2}\right),\quad \gamma_3=\frac{1}{2}\left(a_4-a_2a_3+\frac{a_2^3}{3}\right).
\end{align}
The second order Hankel determinant of $\mathrm{F}_f/2$ is define by
\begin{align}\label{Q-40}
H_{2,1}\left(\frac{\mathrm{F}_f}{2}\right)=\gamma_1\gamma_3-\gamma_2^2=\frac{1}{4}\left(a_2a_4-a_3^2+\frac{1}{12}a_2^4\right),
\end{align}
i.e.,
\begin{align}\label{Q-45}
4H_{2,1}\left(\frac{\mathrm{F}_f}{2}\right)=H_{2,2}(f)+\frac{1}{12}a_2^4.
\end{align}

The Koebe's $1/4$ theorem guarantees that for every $f\in\mathcal{S}$, there exist an inverse function $f^{-1}$ at least on the disk $\mathbb{D}_{1/4}=\{w:|w|<1/4\}$ of the form
\begin{equation}\label{Q-50}
f^{-1}(w)=w+A_2w^2+A_3w^3+\cdots.
\end{equation}
Then by using the identity $f(f^{-1}(w))=w,$ from \eqref{Q-01} and \eqref{Q-50}, we have the coefficient relation
\begin{align*}
A_2=-a_2,~A_3=-a_3+2a_2^2,~A_4=-a_4+5a_2a_3-5a_2^3.
\end{align*}
In this case, the second order Hankel determinant for coefficients of $f^{-1}$ is defined by
\begin{align}\label{Q-55}
H_{2,2}(f^{-1})&=A_2A_4-A_3^2=a_2a_4-a_2^2a_3+a_2^4-a_3^2,
\end{align}
i.e.,
\begin{align}\label{Q-60}
H_{2,2}(f^{-1})=H_{2,2}(f)-a_2^2a_3+a_2^4.
\end{align}

In recent years, estimate of the Hankel determinant whose entries are coefficients of functions of certain class of analytic and univalent functions has been studied by many researchers. Janteng et al. \cite{2007-Janteng-Halim-Darus} obtained sharp estimate of $|H_{2,2}(f)|$ for functions in the classes $\mathcal{S}^*$ and $\mathcal{C}$.
%Zaprawa \cite{2016-Zaprawa} obtained the sharp bound of $|H_{2,2}(f)|$ for typically real functions.
In $2018$, Cho et al. \cite{2018-Cho} obtained the sharp bounds of $|H_{2,1}(f)|$ and $|H_{2,2}(f)|$ for functions in the class $\mathcal{S}^*(\alpha)$. In $2018$, Kowalczyk et al. \cite{2018-Kowalczyk-Lecko-Sim} obtained the sharp inequality $|H_{3,1}(f)|\le4/135$ for $f\in\mathcal{C}$. Lecko et al. \cite{2019-Lecko-Sim-Smiarowska} obtained the sharp estimate of $|H_{3,1}(f)|$ for functions in the class $\mathcal{S}^*(1/2)$.
%In $2022$, Kumar and Verma \cite{2022-SIVAPRASAD} obtained the sharp estimate of $|H_{3,1}(f)|$ for functions in the classes $\mathcal{S}^*(e^z)$ and $\mathcal{C}(e^z)$.
In $2022$, Kowalczyk et al. \cite{2022-Kowalczyk-Thomas} proved the sharp estimate $|H_{3,1}(f)|\le4/9$ for $f\in\mathcal{S}^*$. In $2022$, Kowalczyk et al. \cite{2022-Kowalczyk-Lecko-2} obtained the sharp bound of $|H_{2,1}\left(F_f/2\right)|$ for functions in the classes $\mathcal{S}^*(\alpha)$ and $\mathcal{C}(\alpha)$. Recently, Sim et al. \cite{2022-Sim-Thomas} obtained the sharp bound $|H_{2,2}(f)|$ for functions in the class $\mathcal{C}(\alpha)$ and also obtained the sharp bounds of $|H_{2,2}(f^{-1})|$ for functions in the classes $\mathcal{S}^*(\alpha)$ and $\mathcal{C}(\alpha)$. For more information about the Hankel determinant on  different subclasses of $\mathcal{S}$, we refer to \cite{ 2023-Allu-Arora-Shaji, 2023-Allu-Lecko-Thomas, 2019-N. E. Cho, 2021-Kowalczyk-Lecko,2022-Kowalczyk-Lecko-1, 2018-Kowalczyk-Lecko-Lecko-Sim, 2021-Sim-Zaprawa, 2023-Raza, 2023-Riaz-Raza, 2017-Zaprawa, 2018-Zaprawa, 2021-Zaprawa, 2024-Zaprawa}.\\

The estimation of the successive coefficient $|a_{n+1}|-|a_n|$ plays an important role for general coefficient problem in univalent function theory (see \cite{1983-Duren}). For functions in $\mathcal{S}$, the best estimate till date is
\begin{align}\label{Q-70}
\left||a_{n+1}|-|a_n|\right|\leq A \quad(n=1,2,3,\ldots)
\end{align}
with $A=3.61$ (see \cite{1976-Grinspan}). In 1978, Leung {\cite{1978-Leung}} proved that $A=1$ for $f\in \mathcal{S}^*$. In $2016$, Li and Sugawa \cite{2016-Li-Sugawa} obtained the sharp upper bound of $|a_{n+1}|-|a_n|$ for any $n\in\mathbb{N}$ and the sharp lower bound of $|a_{n+1}|-|a_n|$ for $n=2,3$ for functions in the class $\mathcal{C}$. In $2019$, Arora et al. \cite{2019-V. Arora} estimated the moduli of $|a_{n+1}|-|a_n|$, for functions of $\gamma$-spirallike functions of order $\alpha$.\\

In the present article, we first prove the sharp  estimate of Hankel determinant of  second and third order for functions in the class $\mathcal{S}_u^*$. Also, we obtain the sharp estimate of second order Hankel determinant of logarithmic and inverse coefficients for functions in $\mathcal{S}_u^*$. Finally, we obtain the sharp estimate of the successive coefficients for functions in $\mathcal{S}_u^*$.

\section{Main Results}
Before stating our main results we will discuss some preliminary results, which will help us to prove our results. Let $\mathcal{P}$ denote the class of all analytic functions $p$ in $\mathbb{D}$ with $p(0)=1$ satisfying ${\rm Re\,} p(z)>0$ for $z\in\mathbb{D}$. The following lemma due to Pommerenke \cite{1975-Pommerenke}, Libera and Z\l otkiewicz \cite{1982-Libera-Zlotkiewick,1983-Libera-Zlotkiewick} and Kwon and Lecko \cite{2018-Kwon-Lecko} will be useful.

\begin{lem}\label{Q-75}
If $p\in\mathcal{P}$ is of the form \eqref{Q-15}, then there exist $z_1,z_2,z_3,z_4\in \overline{\mathbb{D}}$ such that
\begin{align}\label{Q-80}
\begin{cases}
p_1&=2z_1,\\
p_2&=2z_1^2+2(1-|z_1|^2)z_2,\\
p_3&=2z_1^3+4(1-|z_1|^2)z_1z_2-2(1-|z_1|^2)z_1z_2^2+2(1-|z_1|^2)(1-|z_2|^2)z_3,\\
p_4&=2z_1^4+2(1-|z_1|^2)(z_1^2z_2^2-3z_1^2z_2+3z_1^2+z_2)z_2\\&\quad +2(1-|z_1|^2)(1-|z_2|^2)(2z_1-2z_1z_2-\bar{z_2}z_3)z_3\\
           &\quad+2(1-|z_1|^2)(1-|z_2|^2)(1-|z_3|^2)z_4.
\end{cases}
\end{align}
\end{lem}

If $f\in \mathcal{S}_u^*$ is of the form \eqref{Q-01} then $f_\theta(z)=e^{-i\theta}f(e^{i\theta} z)$, $\theta\in\mathbb{R}$, called the rotation of $f$, is also belongs to $\mathcal{S}_u^*$. Thus the class $\mathcal{S}_u^*$ is rotationally invariant. Further, it is an easy exercise to verify that
\begin{align*}
&|H_{2,2}(f)|=|H_{2,2}(f_\theta)|,~~ |H_{3,1}(f)|=|H_{3,1}(f_\theta)|,\\[2mm]
& \left|H_{2,2}(f^{-1})\right|=\left|H_{2,2}(f_\theta^{-1})\right|, \left|H_{2,1}\left(\frac{\mathrm{F}_f}{2}\right)\right|=\left|H_{2,1}\left(\frac{\mathrm{F}_{f_\theta}}{2}\right)\right|.
\end{align*}
Therefore, the functionals $|H_{2,2}(f)|, |H_{2,2}(f)|, |H_{2,2}(f^{-1})|$ and $|H_{2,1}(\frac{\mathrm{F}_f}{2})|$ are also rotationally invariant.\\

If $f\in \mathcal{S}_u^*$ then there exists a function $p\in\mathcal{P}$ such that
\begin{equation*}
\frac{zf'(z)}{f(z)}-1=\frac{p(z)-1}{p(z)+1}.
\end{equation*}
 Comparing the coefficients of $z^n,~n=2,3,4$ and $5$, we get the following relations
\begin{equation}\label{Q-85}
a_2=\frac{p_1}{2},\quad a_3=\frac{p_2}{4},\quad a_4=\frac{p_3}{6}-\frac{p_1p_2}{24},\quad a_5=\frac{1}{8}\left(p_4-\frac{p_2^2}{4}-\frac{p_1p_3}{3}+\frac{p_1^2p_2}{12}\right).
\end{equation}

%\begin{lem}\label{Q-90}
%Let $f\in\mathcal{S}_u^*$ be of the form \eqref{Q-01}. Then
%\begin{align*}
%\Theta_1(f)=&|a_2a_4-a_3^2|,\\\Theta_2(\frac{\mathrm{F}_f}{2})=&\frac{1}{4}\left|a_2a_4-a_3^2+\frac{1}{12}a_2^4\right|,\\\Theta_2(f^{-1})=&\left|a_2a_4-a_2^2a_3+a_2^4-a_3^2\right|,\\\Theta_4(f)=&|(a_3a_5-a_4^2)-a_2(a_2a_5-a_3a_4)+a_3(a_2a_4-a_3^2)|
%\end{align*}
%are rotationally invariant $i.e.$, $$
%\Theta_1(f)=\Theta_1(f_\theta),~\Theta_2(\frac{\mathrm{F}_f}{2})=\Theta_2(\frac{\mathrm{F}_{f_\theta}}{2}),~\Theta_3(f^{-1})=\Theta_3(f^{-1}_\theta)~ \text{and}~ \Theta_4(f)=\Theta_4(f_\theta)
%$$
%for every $\theta\in\mathbb{R}$.
%\end{lem}
%Thus Lemma \ref{Q-90} makes it possible to assume that when estimating $\Theta_1(f),~\Theta_2(\frac{\mathrm{F}_f}{2}),~\Theta_2(f^{-1})$ and $\Theta_4(f)$, one selected coefficient of $f$ is a non-negative real number.\\

%%%%%%%%%%%%%%%%%%%%%%%%%%%%%%%%%%%%%%%%%%%%%%%%%%%%%%%%%%%%%%%%%%%%%%%%%%%%%%%%%%%%%%%%%%%%%%%%%%%%%%%%%%%%%%%%%%%%%%%%%%%
If $f\in \mathcal{S}_u^*$ is of the form \eqref{Q-01}, then from (\ref{Q-85}) and using \cite[Lemma 1]{1992-Ma-Minda}, we have
\begin{align*}
|H_{2,1}(f)|=&|a_3-a_2^2|=\frac{1}{4}\left|p_2-p_1^2\right|\le\frac{1}{2}.
\end{align*}
The equality occurs in the above estimate for the function $f_2(z)=ze^z$. In the following theorem we obtain the sharp bound for the Hankel determinant $|H_{2,2}(f)|$ for functions in the class $\mathcal{S}_u^*$.

%\begin{thm}
%Let $f\in \mathcal{S}_u^*$ be of the form \eqref{Q-01}. Then the sharp inequality
%$$\left|H_{2,1}(f)\right|\leq\frac{1}{2}.$$
%\end{thm}
%\begin{proof}
%Let $f\in \mathcal{S}_u^*$ be of the form \eqref{Q-01}. Since the class $\mathcal{S}_u^*$ is rotationally invariant, without loss of generality we take $a_2=\frac{p_1}{2}=z_1\ge0$. The second order Hankel determinant for coefficients of $f$ is
%$$H_{2,1}(f)=
%\begin{vmatrix}
%a_1 & a_2\\
%a_2 &a_3
%\end{vmatrix}
%=a_1a_3-a_2^2=a_3-a_2^2.
%$$\\
%Using the relations in \eqref{Q-90}, we obtain
%\begin{align*}
%|H_{2,1}(f)|=&|a_3-a_2^2|=\left|\frac{p_2}{4}-\frac{p_1^2}{4}\right|.
%\end{align*}
%Therefore, from Lemma \ref{Q-75}, there exists $z_1,z_2\in\mathbb{D}$ such that
%\begin{align*}
%|H_{2,1}(f)|=&\frac{1}{4}|2z_1^2+2(1-z_1^2)z_2-(2z_1)^2|\nonumber\\\leq& \frac{1}{4}|2(1-z_1^2)|z_2|+2z_1^2|\le\frac{1}{2}.
%\end{align*}
%The equality occurs for the function $f_0(z)=ze^z$.
%\end{proof}

\begin{thm}
Let $f\in \mathcal{S}_u^*$ be of the form \eqref{Q-01}. Then
$$
\left|H_{2,2}(f)\right|\leq\frac{1}{4}
$$
and the estimate is sharp.
\end{thm}

\begin{proof}
Let $f\in \mathcal{S}_u^*$ be of the form \eqref{Q-01}. Then using the relations \eqref{Q-85}, we obtain
\begin{align*}
|H_{2,2}(f)|&=|a_2a_4-a_3^2|\\&=\frac{1}{48}|4p_1p_3-p_1^2p_2-3p_2^2|.
\end{align*}
Since the functional $|H_{2,2}(f)|$ is rotationally invariant, without loss of generality we may take $a_2=\frac{p_1}{2}\ge0$. On the other hand, from Lemma \ref{Q-75}, there exists $z_1,z_2,z_3\in\mathbb{\overline{D}}$ with $z_1\ge 0$ such that
\begin{align*}
|H_{2,2}(f)|&=\frac{1}{12}\left|4(1-|z_2|^2)(1-z_1^2)z_1z_3-z_1^4-(3+z_1^2)(1-z_1^2)z_2^2\right|\\
&\le\frac{1}{12}\left(4(1-|z_2|^2)(1-z_1^2)z_1|z_3|+z_1^4+(3+z_1^2)(1-z_1^2)|z_2|^2\right)
\\&\le\frac{1}{12}\left(4(1-|z_2|^2)(1-z_1^2)z_1+z_1^4+(3+z_1^2)(1-z_1^2)|z_2|^2\right)\\&=\frac{1}{12}\left(z_1^4+(1-z_1^2)(4z_1+(z_1-3)(z_1-1)|z_2|^2)\right)\\
&\le\frac{1}{12}\left(z_1^4+(1-z_1^2)(4z_1+(z_1-3)(z_1-1))\right)\\&=\frac{3-2z_1^2}{12}\leq\frac{1}{4}.
\end{align*}
The estimate is sharp for the function  $f_3(z)=ze^\frac{z^2}{2}$.

\end{proof}
%%%%%%%%%%%%%%%%%%%%%%%%%%%%%%%%%%%%%%%%%%%%%%%%%%%%%%%%%%%%%%%%%%%%%%%%%%%%%%%%%%%%%%%%%%%%%%%%%%%%%%%%%%%%%%%%%%%%%%%%%%%%%
\begin{thm}
Let $f\in \mathcal{S}_u^*$ be of the form \eqref{Q-01}. Then
$$\left|H_{3,1}(f)\right|\leq\frac{1}{9}$$
and the estimate is sharp.
\end{thm}

\begin{proof}
Let $f\in \mathcal{S}_u^*$ be of the form \eqref{Q-01}. Then using the relations  \eqref{Q-85}, we obtain
\begin{align}\label{Q-95}
H_{3,1}(f)
&=(a_3a_5-a_4^2)-a_2(a_2a_5-a_3a_4)+a_3(a_2a_4-a_3^2)\\
&=\frac{p_2}{32}\left(p_4-\frac{p_2^2}{4}-\frac{p_1p_3}{3}+\frac{p_1^2p_2}{12}\right)-\left(\frac{p_3}{6}-\frac{p_1p_2}{24}\right)^2\nonumber\\
&\quad -\frac{p_1}{2}\left[\frac{p_1}{16}\left(p_4-\frac{p_2^2}{4}-\frac{p_1p_3}{3} +\frac{p_1^2p_2}{12}\right)-\frac{p_2}{4}\left(\frac{p_3}{6}-\frac{p_1p_2}{24}\right)\right]\nonumber\\
&\quad +\frac{p_2}{4}\left[\frac{p_1}{2}\left(\frac{p_3}{6}-\frac{p_1p_2}{24}\right)-\left(\frac{p_2}{4}\right)^2\right]\nonumber\\
&=\frac{p_4}{32}\left(p_2-p_1^2\right)-\frac{3p_2^2}{128}+\frac{13p_1p_2p_3}{288}-\frac{p_1^2p_2^2}{576}-\frac{p_3^2}{36} +\frac{p_1^3p_3}{96}-\frac{p_1^4p_2}{384}.\nonumber
\end{align}
Since the functional $|H_{3,1}(f)|$ is rotationally invariant, without loss of generality we may take $a_2=\frac{p_1}{2}\ge0$. On the other hand, from Lemma \ref{Q-75}, there exists $z_1,z_2,z_3,z_4\in\mathbb{\overline{D}}$ with $z_1\ge 0$ such that
\begin{align*}
\frac{p_4}{4}\left(p_2-p_1^2\right)
&=-z_1^6-2z_1^4z_2(1-z_1^2)+2z_1^2z_2^2(1-z_1^2)+z_2^3(1-z_1^2)(1-4z_1^2+2z_1^4)\nonumber\\
&\quad + z_1^2z_2^4(1-z_1^2)^2-2z_3(1-z_1^2)(1-|z_2|^2)\left(z_1^3-z_1z_2+z_1z_2^3(1-z_1^2)\right)\nonumber\\
&\quad +z_3^2(1-z_1^2)(1-|z_2|^2)\left(z_1^2\bar{z_2}-|z_2|^2(1-z_1^2)\right)\nonumber\\
&\quad -z_4(1-z_1^2)(1-|z_2|^2)(1-|z_3|^2)\left(z_1^2-z_2(1-z_1^2)\right),\nonumber\\
\frac{p_1p_2p_3}{8}
&=z_1^6+3z_1^4z_2(1-z_1^2)+z_2^2(1-z_1^2)(2z_1^2-3z_1^4)-z_1^2z_2^3(1-z_1^2)^2\nonumber\\
&\quad +z_3(1-z_1^2)(1-|z_2|^2)\left(z_1^3+z_1z_2(1-z_1^2)\right),\\
\frac{p_1^3p_3}{16}
&=z_1^6+2z_1^4z_2(1-z_1^2)-z_1^4z_2^2(1-z_1^2)+z_1^3z_3(1-z_1^2)(1-|z_2|^2),\nonumber\\
\frac{p_3^2}{4}
&= z_1^6+4z_1^4z_2(1-z_1^2)+2z_2^2(1-z_1^2)(2z_1^2-3z_1^4)-4z_1^2z_2^3(1-z_1^2)^2\nonumber\\
&\quad +z_1^2z_2^4(1-z_1^2)^2 +(1-z_1^2)^2(1-|z_2|^2)^2z_3^2 \nonumber\\
&\quad +2z_3(1-z_1^2)(1-|z_2|^2)\left(z_1^3+2z_1z_2(1-z_1^2)-z_1z_2^2(1-z_1^2)\right),\nonumber\\
\frac{p_2^3}{8}& =z_1^6+3z_1^4z_2(1-z_1^2)+3z_1^2z_2^2(1-z_1^2)^2+z_2^3(1-z_1^2)^3,\nonumber\\
\frac{p_1^2p_2^2}{16}& =z_1^6+2z_1^4z_2(1-z_1^2)+z_1^2z_2^2(1-z_1^2)^2,\nonumber\\
\frac{p_1^4p_2}{32}& =z_1^6+z_1^4z_2(1-z_1^2).\nonumber
\end{align*}
Substituting these values in \eqref{Q-95} and then simplifying, we get
\begin{align}\label{Q-100}
H_{3,1}(f)
&=-\frac{z_1^6}{144}+(1-z_1^2)\left(\frac{z_1^4z_2}{48}+\frac{9-z_1^2}{144}z_1^2z_2^2-\frac{z_1^4+2z_1^2+3}{48}z_2^3 \right.\\
&  +\frac{1-z_1^2}{72}z_1^2z_2^4-\frac{1-|z_2|^2}{36} \left(2z_1^2+(6+3z_1^2)z_2-(1-z_1^2)z_2^2\right)z_1z_3\nonumber\\
&+\frac{9z_1^2\bar{z_2}-(1-z_1^2)(8+|z_2|^2)}{72}(1-|z_2|^2)z_3^2\nonumber\\
& \left.-\frac{z_4}{8}(1-|z_2|^2)(1-|z_3|^2)\left(z_1^2-z_2(1-z_1^2)\right)\right)\nonumber\\
&=\frac{1}{144}\left[h_1(z_1,z_2)+h_2(z_1,z_2)z_3+h_3(z_1,z_2)z_3^2+(1-|z_3|^2)h_4(z_1,z_2)z_4\right],\nonumber
\end{align}
where
\begin{align*}
\begin{cases}
h_1(z_1,z_2)&=-z_1^6+z_1^4z_2(1-z_1^2)+z_1^2z_2^2(1-z_1^2)(9-z_1^2)\\
&\quad\quad-3z_2^3(1-z_1^2)(z_1^4+2z_1^2+3)
+2z_1^2z_2^4(1-z_1^2)^2,\nonumber\\
h_2(z_1,z_2)&=-4(1-z_1^2)(1-|z_2|^2)\bigl(2z_1^3+z_1z_2(6+3z_1^2)-z_1z_2^2(1-z_1^2)\bigr),\nonumber\\
h_3(z_1,z_2)&=2(1-z_1^2)(1-|z_2|^2)\bigl(9z_1^2\bar{z_2}-(1-z_1^2)(8+|z_2|^2)\bigr),\nonumber\\
h_4(z_1,z_2,z_3)&=18(1-z_1^2)(1-|z_2|^2)\bigl(-z_1^2+z_2(1-z_1^2)\bigr).\nonumber
\end{cases}
\end{align*}
Since $z_3,z_4\in\overline{\mathbb{D}}$, it follows from \eqref{Q-100} that
\begin{align}\label{Q-105}
|H_{3,1}(f)|&\leq\frac{1}{144}\left[|h_1(z_1,z_2)|+|h_2(z_1,z_2)||z_3|+|h_3(z_1,z_2)||z_3|^2+(1-|z_3|^2)|h_4(z_1,z_2)|\right]\\&\leq \frac{1}{144}\max\left\{\Gamma(x,y,u):(x,y,u)\in\mathrm{D}=[0,1]^3\right\},\nonumber
\end{align}
where
\begin{align*}
\Gamma(x,y,u)=\Gamma_1(x,y)+\Gamma_4(x,y)+\Gamma_2(x,y)u+(\Gamma_3(x,y)-\Gamma_4(x,y))u^2,
\end{align*}
with
\begin{align*}
\Gamma_1(x,y)&=x^6+x^4(1-x^2)y+x^2(1-x^2)(9-x^2)y^2\\
&\quad+3(1-x^2)(x^4+2x^2+3)y^3+2x^2(1-x^2)^2y^4,\\
\Gamma_2(x,y)&=4x(1-x^2)(1-y^2)\left(2x^2+(6+3x^2)y+(1-x^2)y^2\right),\\
\Gamma_3(x,y)&=2(1-x^2)(1-y^2)\left(9x^2y+(1-x^2)(8+y^2)\right),\\
\Gamma_4(x,y)&=18(1-x^2)(1-y^2)\left(x^2+(1-x^2)y\right).
\end{align*}
Since $(x,y,u)\in\mathrm{D}=[0,1]^3$, it follows that, $\Gamma_i(x,y)\geq 0$, for $i=1,2,3$ and $4$. Now, we will complete the proof by considering two  different cases.\\

\textbf{Case-1}: Let $\Gamma_3(x,y)>\Gamma_4(x,y).$ Then
\begin{align}
\Gamma(x,y,u)&=\Gamma_1(x,y)+\Gamma_4(x,y)+\Gamma_2(x,y)u+(\Gamma_3(x,y)-\Gamma_4(x,y))u^2\nonumber\\
&\leq \Gamma_1(x,y)+\Gamma_2(x,y)+\Gamma_3(x,y)\nonumber\\
&=x^6-8x^5+16x^4+8x^3-32x^2+16\nonumber\\
&\quad\quad-(x^6+12x^5+17x^4+12x^3-18x^2-24x)y\nonumber\\
&\quad\quad+(x^6+12x^5-24x^4-16x^3+37x^2+4x-14)y^2\nonumber\\
&\quad\quad-(3x^6-12x^5-15x^4-12x^3+21x^2+24x-9)y^3\nonumber\\
&\quad\quad+(2x^6-4x^5-6x^4+8x^3+6x^2-4x-2)y^4\nonumber\\
&=R(x,y)\nonumber.
\end{align}
To find the maximum of $R(x,y)$, we first find the critical points of $R(x,y)$. A simple calculation gives
\begin{align}
\frac{\partial R}{\partial x}=&-4y^4-24y^3+4y^2+24+2(6y^4-21y^3+37y^2+18y-32)x\nonumber\\&+12(2y^4+3y^3-4y^2-3y+2)x^2-4(6y^4-15y^3+24y^2+17y-16)x^3\nonumber\\&-20(y^4-3y^3-3y^2+3y+2)x^4+6(2y^4-3y^3+y^2-y+1)x^5,\nonumber\\
\frac{\partial R}{\partial y}=&-8y^3+27y^2-28y-8(2y^3-9y^2-y-3)x\nonumber\\&+(24y^3-63y^2+74+18)x^2+4(8y^3+9y^2-8y-3)x^3\nonumber\\&-(24y^3-45y^2+48y+17)x^4\nonumber\\&-4(4y^3-9y^2-6y+3)x^5+(8y^3-9y^2+2y-1)x^6\nonumber.
\end{align}
The numerical solutions of $\frac{\partial R}{\partial x}=0,~\frac{\partial R}{\partial y}=0$ are given by
\begin{align*}
&(x=0,~y=0),~(x\approx -0.72712,~y\approx -20.692),~(x\approx 1.33674,~y\approx 16.0303),\\&(x\approx -1.70703,~y\approx 4.22166),~(x\approx 1.001,~y\approx -1.18779),\\&(x\approx -1.001,~y\approx 0.743482),~(x\approx -0.157665,~y\approx 0.966163),\\&(x\approx -0.871239,~y\approx 0.326156),~(x\approx 1.001,~y\approx -0.091409),\\&(x\approx 1.001,~y\approx 1.2792),~(x\approx -0.960902,~y\approx -0.548902).
\end{align*}
Thus, $R$ has no critical points in $(0,1)\times(0,1)$. We now consider the following subcases.\\

\textbf{Subcase-1a}: Let $x=0$ and $0\leq y\leq1$. Then $R(0,y)=16-14y^2+9y^3-2y^4.$ Since,
$$
R'(0,y)=-2y\left(\left((2y-\frac{27}{8}\right)^2+\frac{167}{64}\right)\le 0 \quad \text{for}~~0\le y\le 1,
$$
$R(0,y)$ is decreasing and  so, $R(0,y)\leq R(0,0)=16.$\\

\textbf{Subcase-1b}: Let $x=1$ and $0\leq y\leq1$. Then $R(1,y)=1.$\\

\textbf{Subcase-1c}: Let $y=0$ and $0\leq x\leq1$. Then
$$
R(x,0)=\left(x^3-4 x^2+4\right)^2\le R(0,0)=16$$

\textbf{Subcase-1d}: Let $y=1$ and $0\leq x\leq1$. Then
$$
R(x,1)=-16x^4+8x^2+9\le R\left(\frac{1}{2},1\right)=10.
$$

\textbf{Case-2}: Let $\Gamma_3(x,y)\leq\Gamma_4(x,y)$. Then
\begin{align*}
\Gamma(x,y,u)&\leq\Gamma_1(x,y)+\Gamma_4(x,y)+\Gamma_2(x,y)u\nonumber\\
&\leq\Gamma_1(x,y)+\Gamma_4(x,y)+\Gamma_2(x,y)\nonumber\\
&=18x^2+8x^3-18x^4-8x^5+x^6\\
&\quad+(18+24x-36x^2-12x^3+19x^4-12x^5-x^6)y\\
&\quad+(4x-9x^2-16 x^3+8x^4+12x^5+x^6)y^2\\
&\quad+(-9-24x+33x^2+12x^3-21x^4+12x^5-3x^6)y^3\\
&\quad+(-4x+2x^2+8x^3-4x^4-4x^5+2x^6)y^4\\&= S(x,y).
\end{align*}
To find the maximum of $S(x,y)$, we first find the critical points of $S(x,y)$. A simple calculation gives
\begin{align*}
\frac{\partial S}{\partial x}=~&24y(1-y^2)+ 4y^2(1-y^2)\\&+x(18y^2-6y^3+4y^4+36(1-y^2)-72y(1-y^2))\\
&+x^2(24(1-y^2)-36y(1-y^2)-24y^2(1-y^2))\\
&+x^3(4y-40y^2-12y^3-16y^4-72(1-y^2)+72y(1-y^2))\\
&+x^4(-40(1-y^2)-60y(1-y^2)+20y^2(1-y^2))\\&+x^5(6-6y+6y^2-18y^3+12y^4),\\
\frac{\partial S}{\partial y}=&-9y^2+18(1-y^2)\\&+x(-48y^2-8y^3+24(1-y^2)+8y(1-y^2))\\
&+x^2(-18y+63y^2+8y^3-36(1-y^2))\\
&+x^3(-16y+24y^2+16y^3-12(1-y^2)-16y(1-y^2))\\
&+x^4(1+16y-45y^2-16y^3+18(1-y^2))\\
&+x^5(16y+24y^2-8y^3-12(1-y^2)+8y(1-y^2))\\
&+ x^6(-1+2y-9y^2+8y^3),
\end{align*}
The numerical solution of $\frac{\partial S}{\partial x}=0,~\frac{\partial S}{\partial y}=0$ in $(0,1)\times(0,1)$ is $(x\approx0.529019 ,y\approx0.681474)$. But at this point $\Gamma_3(x,y)-\Gamma_4(x,y)\approx 0.676099>0$. Thus, $S$ has no critical points. We now consider the following subcases.\\

%$$
%\frac{\partial^2 S}{\partial x^2}\approx -106.965<0 \\ \quad \text{and} \quad \frac{\partial^2 S}{\partial x^2}\frac{\partial^2 S}{\partial y^2}-\frac{\partial^2 S}{\partial x\partial y}\approx 3598.45>0
%.$$ Thus $ S(x,y)$ has maximum value at $(x\approx0.529019 ,y\approx0.681474)$, and so  $$ S(x,y)\leq S(0.529019,0.681474)\approx12.1998.$$
%Now, we find the maximum value of $ S(x,y)$ on the boundary of  $[0,1]\times[0,1]$. For this, we consider the following subcases:\\

\textbf{Subcase-2a}: Let $x=0$ and $0\leq y\leq1$. Then
$$ S(0,y)=9y(2-y^2)\leq S\left(0,\frac{2}{3}\right)=\frac{28}{3}\approx 9.333.$$

\textbf{Subcase-2b}: Let $x=1$ and $0\leq y\leq1$. Then $ S(1,y)=1$.\\

\textbf{Subcase-2c}: Let $y=0$ and $0\leq x\leq1$. Then
$$
 S(x,0)=18x^2(1-x^2)+8x^3(1-x^2)+x^6 \le 18\times \frac{1}{4} +8\times \frac{1}{4}+1= \frac{19}{2}.
$$
where we have used the fact that $x^2(1-x^2)\le1/4$ in $[0,1]$.\\

\textbf{Subcase-2d}: Let $y=1$ and $0\leq x\leq1$. Then $$ S(x,1)=9+8x^2-16x^4\leq S\left(\frac{1}{2},1\right)=10.$$

By combining above two cases, the maximum value of $\Gamma(x,y,u)$ is $16$. Hence from \eqref{Q-105} we get
\begin{align*}
|H_{3,1}(f)|&\leq\frac{16}{144}=\frac{1}{9}.
\end{align*}
The above inequality is sharp for the function $f_4(z)=ze^{\frac{z^3}{3}}$.

\end{proof}

%%%%%%%%%%%%%%%%%%%%%%%%%%%%%%%%%%%%%%%%%%%%%%%%%%%%%%%%%%%%%%%%%%%%%%%%%%%%%%%%%%%%%%%%%%%%%%%%%%%%%%%%%%%%%%%%%%%%%%%%%%%%%%%%%%%%%%%%%%%%%%%%%%%%%%%%%%%%%%%%%%%%%%%%%%%

In our next theorem, we obtain the sharp bound of Hankel determinant for the logarithmic coefficients of functions in the class $\mathcal{S}_u^*$.

\begin{thm}
Let $f\in \mathcal{S}_u^*$ be of the form \eqref{Q-01} and $\mathrm{F}_f$ is given by \eqref{Q-30}. Then
$$\left|H_{2,1}\left(\frac{\mathrm{F}_f}{2}\right)\right|\leq\frac{1}{16},$$
and the estimate is sharp.
\end{thm}

\begin{proof}
Let $f\in \mathcal{S}_u^*$ be of the form \eqref{Q-01}. From \eqref{Q-40}, we have
$$\left|H_{2,1}\left(\frac{\mathrm{F}_f}{2}\right)\right| =|\gamma_1\gamma_3-\gamma_2^2|=\frac{1}{4}\left|a_2a_4-a_3^2+\frac{1}{12}a_2^4\right|.$$
Then using the relations \eqref{Q-85}, we obtain
\begin{align*}
\left|H_{2,1}\left(\frac{\mathrm{F}_f}{2}\right)\right|
&=\frac{1}{4}\left|\frac{p_1p_3}{12}-\frac{p_1^2p_2}{48}-\frac{p_2^2}{16}+\frac{p_1^4}{192}\right|.
\end{align*}
Since the functional $\left|H_{2,1}\left(\frac{\mathrm{F}_f}{2}\right)\right|$ is rotationally invariant, without loss of generality we may take $a_2=\frac{p_1}{2}\ge0$. On the other hand, from Lemma \ref{Q-75}, there exists $z_1,z_2,z_3\in\mathbb{\overline{D}}$ with $z_1\ge 0$ such that
\begin{align*}
\left|H_{2,1}\left(\frac{\mathrm{F}_f}{2}\right)\right|
&=\frac{(1-z_1^2)}{12}\left|(1-|z_2|^2)z_1z_3+\frac{z_1^2+3}{4}z_2^2\right|\\
&\le\frac{(1-z_1^2)}{12}\left((1-|z_2|^2)z_1|z_3|+\frac{z_1^2+3}{4}|z_2|^2\right)\\
&\le\frac{(1-z_1^2)}{12}\left((1-|z_2|^2)z_1+\frac{z_1^2+3}{4}|z_2|^2\right)\\
&=\frac{(1-z_1^2)}{12}\left((z_1+\frac{(3-z_1)(1-z_1)}{4}|z_2|^2\right)\\
&\le\frac{(3+z_1^2)(1-z_1^2)}{48}\le\frac{1}{16}.
\end{align*}
The above estimate is sharp for the function $f_3(z)=ze^\frac{z^2}{2}$.
\end{proof}

In the next theorem, we obtain the sharp bound of the Hankel determinant for the inverse functions of functions in the class $\mathcal{S}_u^*$.

\begin{thm}
Let $f\in \mathcal{S}_u^*$ be of the form \eqref{Q-01} and $f^{-1}$ be the inverse of $f$ of the form \eqref{Q-50}. Then
$$\left|H_{2,2}(f^{-1})\right|\leq\frac{5}{12},$$
and the estimate is sharp.
\end{thm}

\begin{proof}
Let $f\in \mathcal{S}_u^*$ be of the form \eqref{Q-01}. From \eqref{Q-55}, we have
$$
\left|H_{2,2}(f^{-1})\right|=\left|A_2A_4-A_3^2\right|=\left|a_2a_4-a_2^2a_3+a_2^4-a_3^2\right|.
$$
By using the relations \eqref{Q-85}, we obtain
$$
\left|H_{2,2}(f^{-1})\right|=\left|\frac{p_1p_3}{12}-\frac{p_1^2p_2}{12}+\frac{p_1^4}{16}-\frac{p_2^2}{16}\right|
$$
Since the functional $|H_{2,2}(f^{-1})|$ is rotationally invariant, without loss of generality we may take $a_2=\frac{p_1}{2}\ge0$. On the other hand, from Lemma \ref{Q-75}, there exists $z_1,z_2,z_3\in\mathbb{\overline{D}}$ with $z_1\ge 0$ such that
\begin{align*}
\left|H_{2,2}(f^{-1})\right|\le&\frac{5}{12}z_1^4+\frac{(1-z_1^2)}{12}\left(4(1-|z_2|^2)z_1|z_3|+6z_1^2|z_2|+(3+z_1^2)|z_2|^2\right)\\
\le&\frac{5}{12}z_1^4+\frac{(1-z_1^2)}{12}\left(4(1-|z_2|^2)z_1+6z_1^2|z_2|+(3+z_1^2)|z_2|^2\right)\\
%=&\frac{5}{12}z_1^4+\frac{(1-z_1^2)}{12}\left(4z_1+6z_1^2|z_2|+(3+z_1^2-4z_1)|z_2|^2\right)\\
=&\frac{5}{12}z_1^4+\frac{(1-z_1^2)}{12}\left(4z_1+6z_1^2|z_2|+(1-z_1)(3-z_1)|z_2|^2\right)\\
\le&\frac{5}{12}z_1^4+\frac{(1-z_1^2)}{12}\left(4z_1+6z_1^2+(1-z_1)(3-z_1)\right)\\
%=&\frac{5}{12}z_1^4+\frac{(1-z_1^2)(3+7z_1^2)}{12}\\
=&\frac{1}{12}(3+4z_1^2-2z_1^4)\\
\le&\frac{5}{12}.
\end{align*}
The above inequality is sharp for the function $f_2(z)=ze^z$.
\end{proof}

The relations between Hankel determinant of the function $f$, logarithmic function $\mathrm{F}_f$ and inverse function $f^{-1}$ when $f\in\mathcal{S}_u^*$ are as follows:\\

\noindent\textbf{(i)}   From \eqref{Q-45}, we get
\begin{align*}
\left|4H_{2,1}\left(\frac{\mathrm{F}_f}{2}\right)-H_{2,2}(f)\right|=\frac{1}{12}|a_2|^4=\frac{1}{12}\left|\left(\frac{p_1}{2}\right)^4\right|\le\frac{1}{12},
\end{align*}
and the inequality is sharp.
The extremal function is $f_2(z)=ze^z$ for which $H_{2,1}\left(\frac{\mathrm{F}_f}{2}\right)=0$ and $H_{2,2}(f)=\frac{1}{12}$.\\

\noindent\textbf{(ii)}   From \eqref{Q-60}, we get
\begin{align*}
\left|H_{2,2}(f^{-1})-H_{2,2}(f)\right|=|a_2|^2|a_3-a_2^2|=\frac{|p_1|^2|p_2-p_1^2|}{16}\le\frac{1}{2},
\end{align*}
and the inequality is sharp.
The equality holds for the function $f_2(z)=ze^{z}$ for which $H_{2,2}(f^{-1})=\frac{5}{12}$ and $H_{2,2}(f)=-\frac{1}{12}$.\\

%%%%%%%%%%%%%%%%%%%%%%%%%%%%%%%%%%%%%%%%%%%%%%%%%%%%%%%%%%%%%%%%%%%%%%%%%%%%%%%%%%%%%%%%%%%%%%%%%%%%%%%%%%%%%%%%%%%%%%%%%%%%%%%%%%%%%%%%%%%%%%%%%%%%%%%%%%%%%%%%%
In our next theorem, we obtain the estimate of successive coefficients for functions in the class $\mathcal{S}_u^*$.

\begin{thm}\label{Q-110}
Let $f\in \mathcal{S}_u^*$ be of the form \eqref{Q-01}. Then
$$ -\frac{1}{n-1}\leq|a_{n+1}|-|a_n|\leq \frac{1}{n}\quad\text{for}~n\geq 2.$$
Both the inequalities are sharp.
\end{thm}
\begin{proof}
Let $f\in \mathcal{S}_u^*$ be of the form \eqref{Q-01}. Then, from \eqref{Q-15} we have
Since $f(z)=z+\sum_{n=2}^{\infty}a_n z^n\in \mathcal{S}_u^*$ so we have
 $$|a_n|\leq\frac{1}{n-1}\quad\text{for }n\geq 2.$$
 A simple computation gives
 $$ -\frac{1}{n-1}\leq|a_{n+1}|-|a_n|\leq \frac{1}{n}\quad\text{for}~n\geq 2.$$
  Clearly,
$$|a_{n+1}|-|a_n|\leq|a_{n+1}|\leq\frac{1}{n},$$
and
$$
\quad|a_{n+1}|-|a_n|\geq-|a_n|\geq-\frac{1}{n-1}.
$$
The left hand side inequality is sharp for the function $f_n(z)=z\exp(\frac{z^{n-1}}{n-1})$ and the right hand side inequality is sharp for the function $f_{n+1}(z)=z\exp(\frac{z^n}{n})$.

\end{proof}

\vspace{4mm}
\noindent\textbf{Declarations}\\

\noindent\textbf{Conflict of interest:} The authors declare that they have no conflict of interest.\\

\noindent\textbf{Data availability:}
Data sharing not applicable to this article as no data sets were generated or analyzed during the current study.\\

\noindent\textbf{Authors Contributions:}
All authors contributed equally to the investigation of the problem and the order of the authors is given alphabetically according to their surname. All authors read and approved the final manuscript. \\

%\noindent\textbf{Acknowledgement:}
%The authors thank the institute National Institute of Technology Durgapur for granting financial support to pursue research work.

\end{document}